\documentclass[leqno,CJK]{siamltex704}
\usepackage{amssymb,amsmath,graphicx,amscd,mathrsfs}
\usepackage{color,xcolor,amsmath}
\usepackage{amsmath}
\usepackage{graphicx}
\usepackage{mathrsfs}
\usepackage{float}
\usepackage{amsfonts,amssymb}
\usepackage{dsfont}
\usepackage{pifont}
\usepackage{hyperref}
\usepackage{multirow}
\usepackage{placeins}
\usepackage{geometry}
\usepackage[numbers,sort&compress]{natbib}%多篇参考论文的排序
\numberwithin{equation}{section}
\def\3bar{{|\hspace{-.02in}|\hspace{-.02in}|}}
\def\E{{\mathcal{E}}}
\def\T{{\mathcal{T}}}

\def\dQ{{\mathbb{Q}}}

\def\b0{\boldsymbol{0}}
\def\sumT{\sum_{T\in\mathcal{T}_h}}     %new
\def\bn{{\mathbf{n}}}

\def\bf{{\mathbf{f}}}

\newtheorem{algorithm1}{Weak Galerkin Algorithm}

 \newcommand{\norm}[1]{\left\Vert#1\right\Vert}
 
 \newcommand{\trb}[1]{|\!|\!|#1|\!|\!|}
 \newcommand{\pa}[2]{\frac{\partial #1}{\partial #2}}
 \newcommand{\D}[1]{\Delta #1}
 \newcommand{\DW}[1]{\Delta _w #1}
 \newcommand{\G}[1]{\nabla #1}
 
 \newcommand{\p}[1]{\partial #1}
 \newcommand{\normT}[1]{\norm{#1}_T}
 \newcommand{\normPT}[1]{\norm{#1}_{\partial T}}

\allowdisplaybreaks % For long formula

\begin{document}

\title{The stabilizer-free weak Galerkin finite element method for the Biharmonic equation using polynomials of reduced order}

\author{
  Shanshan Gu
  \and
  Qilong Zhai
% \thanks{Department of Mathematics, Jilin University, Changchun, 130012,
% China (diql15@mails.jlu.edu.cn).}
% \and
% Hehu Xie\thanks{LSEC, ICMSEC, Academy of Mathematics and Systems Science,
% Chinese Academy of Sciences, Beijing 100190, P.R. China,  and School of Mathematical Sciences, University
% of Chinese Academy of Sciences, Beijing, 100049, China (hhxie@lsec.cc.ac.cn). The research
% of this author was supported by Science Challenge Project (No. TZ2016002),
% National Natural Science Foundations of China (NSFC 11771434, 91330202, 11371026, 11001259),
% the National Center for Mathematics and Interdisciplinary Science, CAS.}
% \and Ran Zhang\thanks{The research of this author was supported in part by China Natural National Science Foundation (91630201, U1530116,11471141, 11726102), and by the Program for Cheung Kong Scholars of Ministry of Education of China, Key Laboratory of Symbolic Computation and Knowledge Engineering of Ministry of Education, Jilin University, Changchun, 130012, P.R. China. } \and
% Zhimin Zhang\thanks{Beijing Computational Science Research Center, Beijing,
% 100193, China (zmzhang@csrc.ac.cn); Department of Mathematics, Wayne State University,
% Detroit, MI 48202 (zzhang@math.wayne.edu).
% The research of this author was supported in part by the National
% Natural Science Foundation of China (NSFC 11471031, 91430216)
% and the U.S. National Science Foundation (DMS--1419040).}
}
\maketitle
%--------------------------------------------------------------------------------------------
\begin{abstract}
In this article, we decrease the degree of the polynomials on the boundary of the weak functions and modify the definition of the weak laplacian which are introduced in \cite{BiharmonicSFWG} to use the SFWG method for the biharmonic equation. Then we propose the relevant numerical format and obtain the optimal order of error estimates in $H^2$ and $L^2$ norms. Finally, we confirm the estimates using numerical experiments.
\end{abstract}


% \begin{AMS}
% Primary, 65N30, 65N15, 65N12, 74N20; Secondary, 35B45, 35J50, 35J35
% \end{AMS}

\section{Introduction}

In this article, we consider the biharmonic equation of the form
    \begin{align}
    \label{Bmodel-equ1}\Delta^2u&=f, \qquad in~\Omega,\\
    \label{Bmodel-equ2}u&=0,\qquad on ~\partial \Omega,\\
    \label{Bmodel-equ3}\pa{u}{\bn}&=0,\qquad on ~\partial \Omega,
    \end{align}
    where $\bn$ is the outward unit normal vector along $\partial \Omega$ and $\Omega$ is the bounded polygonal or polyhedral domain in $\mathbb{R}^d(d=2,3)$.

    We can give the variational form with ease: find $u\in H^2_0(\Omega)$ such that
    \begin{align}
        \label{var-equ1}(\Delta u,\Delta v)=(f,v), \qquad \forall ~v\in H^2_0(\Omega),
    \end{align}
    where we use the definition of the space
    \begin{align*}
    H^2_0(\Omega)=\{v\in H^2(\Omega):~v|_{\partial\Omega}=0,~\pa{v}{\bn}|_{\partial\Omega}=0\}.
    \end{align*}

    The conforming finite element methods, as traditional techniques, have been utilized to solve the biharmonic equation \cite{MR1191139,BiharmonicCFEM2,MR2817542} based on the above form. They construct $C^1$-continuous finite elements to form the finite dimensional subspaces of $H^2(\Omega)$. However, the complexity of constructing such continuous elements has led attention to other approaches.

    The researchers use variational forms different from the form (\ref{var-equ1}) to avoid the construction of $C^1$-continuous finite elements. For example, the mixed finite element methods introduce auxiliary variables to build the variational formulations. The auxiliary variables introduced usually have some physical significance. Different auxiliary variables are suitable for solving different types of physics problems, such as, introducing $\varphi = -\D u$ in \cite{CRMFEM,BiharmonicMixFEM3} for solving hydrodynamics problems or introducing $\sigma = -\nabla ^2 u$ in \cite{MR0386298,HJMFEM} to solve plate problems.

    In order to avoid constructing the $H^2$ conforming finite elements, there are discontinuous Galerkin finite element methods based on variational form (\ref{var-equ1}), such as IPDG method \cite{BiharmonicDGFEM}, which adopt discontinuous functions sets as finite element spaces to approximate $H^2(\Omega)$. Although it is more flexible by selecting discontinuous functions, the numerical scheme is introduced complicated penalty term.

    In the last decade, a new discontinuous Galerkin finite element method, the weak Galerkin (WG) finite element method, has been well developed. The method use weak functions and weak differential operators in numerical formulation to solve various equations, such as the Possion equation \cite{mu_weak_2012,PossionWG}, the Stokes equation \cite{StokesWG}, the Brinkman equation \cite{BrinkmanWG}, the biharmonic equation \cite{BiharmonicWG,BiharmonicWGMFEM} and so on. Weak functions denote the polynomials inside the cell and at the boundary as $v_0$ and $v_b$, respectively. Weak differential operators are definited by using the partial integrals of differential operators. The subsequent improvement of this method is also carried out through these two parts. 
    
    Scholars reduce the order of $v_b$ to decrease the degrees of freedom to solve the Possion equation in \cite{PossionWGReOrder}, the biharmonic equation in \cite{BiharmonicWGReOrder} and so on. The modified weak Galerkin (MWG) finite element method utilizes the weak functions $\{v_0,\{v_0\}\}$ instead of $\{v_0,v_b\}$ to solve equations \cite{PossionMWG,BiharmonicMWG}. Using the MWG finite element method, the complexity of the solution process is reduced by decreasing the degrees of freedom. In addition, the stabilizers appear in WG numerical formats to ensure weak continuity. The stabilizer free weak Galerkin (SFWG) finite element method can eliminate the stabilizers by increasing the order of polynomials in the range of weak operators, which simplies the numerical scheme. The SFWG methods with weak functions $(P_k(T),P_{k}(e))$ in \cite{PossionSFWG} or $(P_k(T),P_{k-1}(e))$ in \cite{PossionSFWGReOrder} are used to solve the second-order elliptic equation. In \cite{BiharmonicSFWG}, authors use the SFWG method to solve the biharmonic equation by the weak functions $(P_k(T),P_k(e),P_{k-1}(e))$. In this paper, we intend to apply the SFWG method by the weak functions $(P_k(T),P_{k-1}(e),P_{k-1}(e))$ to the biharmonic equation.

    The outline of this article is as follows. In Section 2, we make preparations and propose the numerical scheme. In the next section, we derive the error equations and yield the error estimates in $H^2$ and $L^2$ norms. Then, we utilize two examples to obtain the correctness of the theoretical results in Section 4. In the final section, we summarize the work done in this paper and make plans for the future.

% {\color{blue}
% In \cite{PossionCDG1}, the conforming discontinuous Galerkin finite element method was used to solve the second-order elliptic equation in triangular meshes by RT element. Since the method has both the simple formulation of the conforming finite element method and the flexibility of using discontinuous approximation, it is called the name. Furthermore, the conforming discontinuous Galerkin finite element method was used on polygon/polyhedra meshes in \cite{PossionCDG2}, which used the definition of weak functions in MWG method and the domain space of weak operators in SFWG method.
% }

%======================================================================================
\section{SFWG numerical scheme for the biharmonic equation}
\subsection{Notations for partitions}

Suppose $K$ is an open bounded domain in $\mathbb{R} ^d$ and $s$ is a positive integer. We utilize $\|\cdot\|_{s,K}$,
$|\cdot|_{s,K}$, $(\cdot,\cdot)_{s,K}$ to represent the norm, seminorm and
inner product of Sobolev space $H^s(K)$, respectively. If $K=\Omega$, we drop the subscript $K$ and drop $s$ if $s=2$.

We let the partition $\T_h$ of $\Omega$ satisfy assumptions in \cite{PossionMixedWG} and denote $\E_h$ as the set of all edges in $\mathbb{R}^2$ or flat faces in $\mathbb R ^3$. $\E^0_h=\E_h\backslash\partial\Omega$ is defined as the set of all interior edges or flat faces. We denote the mesh size of $\T_{h}$ by $h$.

In addition, we define the set of normal directions on $\E_h$ as follows
\begin{eqnarray*}
  \mathcal{D} _h=\{\bn _e:\bn_e\text{ is unit and normal to }e,~e\in\E_h\}.
\end{eqnarray*}

To define the SFWG method, we
introduce some discrete spaces as follows.
\begin{align}
  V_h&=\{v=\{v_0,v_b,v_n\bn _e\},v_0|_T\in P_k(T),v_b|_e\in P_{k-1}(e),v_n|_e\in P_{k-1}(e),T\in \mathcal{T}_h,e\in \mathcal{E}_h\},\label{spa-Vh} \\
  V_h^0&=\{v=\{v_0,v_b,v_n\bn_e\}\in V_h:~v_b|_{e}=0,~v_n|_e=0,~
  e\subset\partial\Omega\},\label{spa-Vh0}
\end{align}
where $\rho\in P_{k}(T)$ denotes the polynomial with degree no more than $k$ on the inner part of the element $T$ and $P_k(e)$ represents the set of polynomials with degree no more than $k$ on edge $e$.

We denote by $Q_0$ the $L^2$ projection operator to
$P_{k}(T)$ on each element $T\in\T_h$ and $Q_b$ denotes a locally defined $L^2$ projection operator to $P_{k-1}(e)$ on each edge $e\in\E_h$.

Then we can define the discrete weak laplacian of functions:
\begin{definition}
  For each $v\in V_h+H^2(\Omega)$, $\Delta_w v|_T \in P_{j}(T)$ satisfying
  \begin{eqnarray}\label{def-wlaplace}
    (\Delta_w v,\varphi)_{T}=(\Delta v_0,\varphi)_{T}+\langle Q_b(v_0-v_b),\nabla\varphi\cdot\bn\rangle_{\partial T}-
    \langle (\nabla v_0-v_n\bn_e)\cdot\bn,\varphi\rangle_{\partial T},\quad
    \forall \varphi\in P_{j}(T),
  \end{eqnarray}
  where $j>k$ and $\bn$ denotes the outward unit normal vector.
\end{definition}

% \begin{definition}
%   For any $v\in V_h+H^1(\Omega)$,
%   the discrete weak gradient $\GW v|_{T}\in [P_{j}(T)]^d$ satisfies
%   \begin{align}\label{def-wgradient}
%     (\GW v,\bq)_T = -(v_0,\G\cdot\bq)_T+\langle v_b,\bq\cdot\bn\rangle _{\p T},\qquad \forall\bq\in [P_j(T)]^d.
%   \end{align}    
% \end{definition}

And we use the following simple notations:
\begin{align*}
(v,w)_{\T_h}&=\sumT (v,w)_T=\sumT \int _T vwdT,\\
\langle v,w\rangle_{\partial \T_h}&=\sumT\langle v,w\rangle_{\partial T}=\sumT\int_{\partial T}vwds.
\end{align*}
    
\subsection{Numerical Scheme}
With above preparations, we can define the numerical scheme as follows.
\begin{algorithm1}
The numerical solution of $(\ref{Bmodel-equ1})-(\ref{Bmodel-equ3})$
$u_h\in V^0_h$ such that
\begin{align}
  \label{scheme}(\Delta _w u_h,\Delta _w v)_{\T_h}=(f,v _0)_{\T _h},\qquad\forall v\in V^0_h.
\end{align}
\end{algorithm1}

\begin{lemma}\label{operator-ex}
  For any $v\in H^2(\Omega)$, we have 
  \begin{align}
    \DW v =\dQ_h\D v, \qquad\forall T\in \T_h,
  \end{align}
  where we define $\dQ_{h}$ as the projection to $[P_{j}(T)]^d$
  in $T\in\T_h$.
\end{lemma}
\begin{proof}
  For any $T\in\T_h$, using the definitions of $\DW$ and $\dQ_h$, we have
  \begin{align*}
    (\DW v,\varphi)_T=&(\D v,\varphi)_T+\langle Q_b(v-v),\G\varphi\cdot\bn \rangle_{\p T}-\langle(\G v-(\G v\cdot\bn_e)\bn_e)\cdot\bn,\varphi\rangle _{\p T}\\
    =&(\D v,\varphi)_T\\
    =&(\dQ_h\D v,\varphi)_T
  \end{align*}
  for any $\varphi\in P_j(T)$, which implies $\DW v=\dQ_h\D v$.
\end{proof}

\subsection{Existence and Uniqueness}
First, we introduce the following semi-norms.
\begin{definition}\label{norm}
For $v\in V_h+H^2(\Omega)$,
\begin{align*}
\trb v ^2=&~(\DW v,\DW v)_{\T_h},\\
\|v\|^2_{2,h}=&~\sumT \Big(\norm{\D v_0}^2_T+h_T^{-3}\norm{Q_b(v_0-v_b)}^2_{\partial T}+h_T^{-1}\norm{(\G v_0-v_n\bn_e)\cdot\bn}^2_{\partial T}\Big).
\end{align*}
\end{definition}

\begin{lemma}\label{lemma-norm-equ}
There exist two positive constants $C_1$ and $C_2$ such that
\begin{align*}
C_1\norm{v}_{2,h}\leq\trb v\leq C_2\norm{v}_{2,h},\quad \forall v\in V_h.
\end{align*}
\end{lemma}
The proof of this Lemma is similar with the process of the Lemma 3.3 in \cite{BiharmonicSFWG}.
\begin{lemma}
  $\norm{\cdot}_{2,h}$ is the norm of $V^0_h$.
\end{lemma}
\begin{proof}
  According to the definition of $\norm{\cdot}_{2,h}$, we shall only prove the positivity property. Assume that $v\in V^0_h$ such that $\norm{v}_{2,h}=0$. 
  Then we have
  \begin{align*}
    \D v_0|_{T}=0,~   Q_b(v_0-v_b)|_{\p T}=0,~ (\G v_0-v_n\bn_e)\cdot\bn|_{\p T}=0,\qquad\forall T\in\T_h.
  \end{align*}
  And $(\G v_0-v_n\bn_e)\cdot\bn|_{\p T}=0$ implies ($\G v_0\cdot\bn_e-v_n)|_{\p T}=0$ for any $T\in\T_h$.
  Next we verify that $\nabla v_0=0,~\forall T\in \T_h$. From the Gauss formula, we have
  \begin{align*}
    \norm{\nabla v_0}^2_T=(\nabla v_0,\nabla v_0)_T=-(\Delta v_0,v_0)_T +\langle \nabla v_0\cdot\bn,v_0 \rangle _{\partial T}=\langle \nabla v_0\cdot\bn,v_0 \rangle _{\partial T}.
  \end{align*}

  To sum over all $T$, we have
  \begin{align*}
    \sumT \norm{\nabla v_0}^2_T=\sumT\langle \nabla v_0\cdot\bn,v_0 \rangle _{\partial T}
  \end{align*}

  When $e\in\E_h^0$, assume $T_1$ and $T_2$ be two elements sharing $e$, $v_0^1,v_0^2$ be the values of $v$ on $T_1$ and $T_2$, and $\bn_1,\bn_2$ be the unit outward normal vectors of $T_1,T_2$ on $e$.
  \begin{align*}
    \langle \nabla v_0^1\cdot\bn_1,v_0^1 \rangle _e+\langle \nabla v_0^2\cdot\bn_2,v_0^2 \rangle _e&=\pm \Big(\langle \G v_0^1\cdot\bn_e,v_0^1\rangle _e-\langle \G v_0^2\cdot\bn_e,v_0^2 \rangle _e\Big)\\
    &=\pm \Big(\langle \G v_0^1\cdot\bn_e,Q_bv_0^1\rangle _e-\langle \G v_0^2\cdot\bn_e,Q_bv_0^2 \rangle _e\Big)\\
    &=\pm \Big(\langle v_n,v_b \rangle _e-\langle v_n,v_b \rangle _e\Big)\\
    &=0,
  \end{align*}
  where we use $Q_bv_0|_{\p T}=v_b|_{\p T}$ and $\G v_0\cdot\bn_e|_{\p T}=v_n|_{\p T}$.

  With $\nabla v_0\cdot\bn_e=v_n=0$ on any boundary edge, we get
  \begin{align}
    \sumT \norm{\nabla v_0}^2_T=0,
  \end{align}
  which deduces $\nabla v_0=0$ on any element $T$. Using $Q_bv_0|_e=v_b|_e$, $\forall e\in \E_h$ and $v\in V_h^0$, we obtain $v=0$ on $\Omega$.
\end{proof}

Furthermore, by the norm equivalence of the Lemma $\ref{lemma-norm-equ}$, we have $\trb\cdot$ is the norm in $V_h^0$.

\begin{theorem}
  The numerical scheme (\ref{scheme}) exists an unique solution.
\end{theorem}
\begin{proof}
  Let's say $f=0$ and $v=u_h$ in (\ref{scheme}), then we have $\trb{u_h}^2=0$. Since $\trb{\cdot}$ is the norm of $V_h^0$, we get $u_h=0$.
\end{proof}

%======================================================================================
\section{Error analysis}
\subsection{Error equation}
\begin{theorem}
  Let $e_h=u-u_h$ and $\varepsilon _h=Q_hu-u_h$, then we have
  \begin{align}
    \label{err-equ1}(\DW e_h,\DW v)_{\T _h}=&-l_1(u,v)+l_2(u,v)-l_3(u,v),\quad\forall v\in V_h^0,\\
    \label{err-equ2}(\DW \varepsilon _h,\DW v)_{\T _h}=&-l_1(u,v)+l_2(u,v)-l_3(u,v)+(\DW (Q_hu-u),\DW v),\quad\forall v\in V_h^0,
  \end{align}
  where 
  \begin{align*}
    l_1(u,v)=&~\langle \G(\D u-\dQ_h\D u)\cdot\bn,Q_b(v_0-v_b)\rangle_{\p\T _h},\\
    l_2(u,v)=&~\langle\D u-\dQ_h\D u,(\G v_0-v_n\bn_e)\cdot\bn\rangle _{\p\T _h},\\
    l_3(u,v)=&~\langle\G (\D u)\cdot\bn,v_0-Q_bv_0\rangle _{\p\T _h}.
  \end{align*}
\end{theorem}
\begin{proof}
  For any $v\in V_h^0$, from (\ref{Bmodel-equ1}), we have
  \begin{align*}
    &(f,v_0)_{\T _h}\\
    =&(\D ^2u,v_0)_{\T _h}\\
    =&-(\G (\D u),\G v_0)_{\T _h}+\langle\G (\D u)\cdot\bn,v_0\rangle _{\p\T _h}\\
    =&(\D u,\D v_0)_{\T _h}-\langle \D u,\G v_0\cdot\bn\rangle _{\p \T _h}+\langle\G (\D u)\cdot\bn,v_0-v_b\rangle _{\p\T _h}\\
    =&(\D u,\D v_0)_{\T _h}-\langle \D u,(\G v_0-v_n\bn_e)\cdot\bn\rangle _{\p \T _h}+\langle\G (\D u)\cdot\bn,v_0-Q_bv_0\rangle _{\p\T _h}+\langle\G (\D u)\cdot\bn,Q_bv_0-v_b\rangle _{\p\T _h},
  \end{align*}
  where we use $v_b|_{\p\Omega}=0$ and $v_n|_{\p\Omega}=0$. This implies
  \begin{align*}
    &(\D u,\D v_0)_{\T _h}\\
    =&(f,v_0)_{\T _h}+\langle \D u,(\G v_0-v_n\bn_e)\cdot\bn\rangle _{\p \T _h}-\langle\G (\D u)\cdot\bn,v_0-Q_bv_0\rangle _{\p\T _h}-\langle\G (\D u)\cdot\bn,Q_bv_0-v_b\rangle _{\p\T _h}.
  \end{align*}

  Utilizing the above equation, the Lemma \ref{operator-ex} and the definition of $\DW$, for any $v\in V_h^0$, we obtain
  \begin{align*}
    &(\DW u,\DW v)_{\T _h}\\
    =&(\dQ_h \D u,\DW v)_{\T_h}\\
    =&(\D v_0,\dQ_h\D u)_{\T _h}+\langle Q_b(v_0-v_b),\G(\dQ_h\D u)\cdot\bn\rangle _{\p\T_h}-\langle (\G v_0-v_n\bn_e)\cdot\bn,\dQ_h\D u\rangle _{\p\T_h}\\
    =&(\D v_0,\D u)_{\T _h}+\langle Q_b(v_0-v_b),\G(\dQ_h\D u)\cdot\bn\rangle _{\p\T_h}-\langle (\G v_0-v_n\bn_e)\cdot\bn,\dQ_h\D u\rangle _{\p\T_h}\\
    =&(f,v_0)_{\T _h}-l_1(u,v)+l_2(u,v)-l_3(u,v).
  \end{align*}
  Combining with the numerical scheme (\ref{scheme}), we have
  \begin{align*}
    (\DW (u-u_h),\DW v)_{\T _h}=-l_1(u,v)+l_2(u,v)-l_3(u,v), \quad\forall v\in V_h^0.
  \end{align*}
  Add $(\DW (Q_hu-u),\DW v)_{\T _h}$ to both sides of the above equation, we get (\ref{err-equ2}).
\end{proof}

% \begin{lemma}
%   \cite{PossionMixedWG}For $w\in H^{\max{\{k+1,4\}}}(\Omega)$, $0\leq s\leq 2$, we have
%   \begin{align}
%     \label{Q0-est}\sumT h_T^{2s}\norm{u-Q_0u}_{s,T}^2\leq &Ch^{2(k+1)}\norm{u}_{k+1}^2,\\
%     % \label{Q0-est2}\sumT h_T^{2s}\norm{\D u-\mathbf Q_h\D u}_{s,T}^2\leq &Ch^{2(m-1)}\norm{u}_{m+1}^2,
%     \label{Qh-est}\sumT h_T^{2s}\norm{\D u-\dQ_h\D u}_{s,T}^2\leq &Ch^{k-1+\delta _{k,2}}\norm{w}_{\max{\{k+1,4\}}}.
%   \end{align}
%   %where $\mathbf Q_h$ is the $L^2$ projection to $P_{k-2}(T)$ for any $T\in\T_h$.
% \end{lemma}

\begin{lemma}
  \cite{PossionMixedWG}For $w\in H^{k+2}(\Omega)$, we have
  \begin{align}
    \label{Q0-est}\sumT h_T^{2s}\norm{w-Q_0w}_{s,T}^2\leq &Ch^{2(k+1)}\norm{w}_{k+1}^2,\\
    % \label{Q0-est2}\sumT h_T^{2s}\norm{\D u-\mathbf Q_h\D u}_{s,T}^2\leq &Ch^{2(m-1)}\norm{u}_{m+1}^2,
    \label{Qh-est}\sumT h_T^{2s}\norm{\D w-\dQ_h\D w}_{s,T}^2\leq &Ch^{2k}\norm{w}^2_{k+2}.
  \end{align}
  %where $\mathbf Q_h$ is the $L^2$ projection to $P_{k-2}(T)$ for any $T\in\T_h$.
\end{lemma}

% \begin{lemma}
%   Assume $w\in H^{\max{\{k+1,4\}}}(\Omega)$, and we have the following estimates
%   \begin{align}
%     \Big(\sumT h_T\normPT{\D w-\dQ_h\D w}^2\Big)^{\frac{1}{2}}\leq &Ch^{k-1+\delta _{k,2}}\norm{w}_{\max{\{k+1,4\}}},\\
%     \Big(\sumT h_T^3\normPT{\G(\D w-\dQ_h\D w)}^2\Big)^{\frac{1}{2}}\leq &Ch^{k-1+\delta _{k,2}}\norm{w}_{\max{\{k+1,4\}}}.
%   \end{align}
% \end{lemma}

\begin{lemma}
  Assume $w\in H^{k+2}(\Omega)$, and we have the following estimates
  \begin{align}
    \label{Qh-est-h1}\Big(\sumT h_T\normPT{\D w-\dQ_h\D w}^2\Big)^{\frac{1}{2}}\leq &Ch^{k}\norm{w}_{k+2},\\
    \label{Qh-est-h3}\Big(\sumT h_T^3\normPT{\G(\D w-\dQ_h\D w)}^2\Big)^{\frac{1}{2}}\leq &Ch^{k}\norm{w}_{k+2}.
  \end{align}
\end{lemma}
\begin{proof}
  By the trace inequality and the (\ref{Q0-est})-(\ref{Qh-est}), we can get
  \begin{align*}
    \Big(\sumT h_T\normPT{\D w-\dQ_h\D w}^2\Big)^{\frac{1}{2}}\leq &\Big(\sumT \normT{\D w-\dQ_h\D w}^2+h_T^2\normT{\G (\D w-\dQ_h\D w)}^2\Big)^{\frac{1}{2}}\\
    \leq &Ch^{k}\norm{w}_{k+2}.
  \end{align*}
  \begin{align*}
    \Big(\sumT h_T^3\normPT{\G(\D w-\dQ_h\D w)}^2\Big)^{\frac{1}{2}}\leq &\Big(\sumT h_T^2\normT{\G(\D w-\dQ_h\D w)}^2+h_T^4\normT{\G^2(\D w-\dQ_h\D w)}^2\Big)^{\frac{1}{2}}\\
    \leq &Ch^{k}\norm{w}_{k+2}.
  \end{align*}
  
\end{proof}

From the proof of the Lemma A.8 in \cite{BiharmonicWGReOrder}, we have the following estimates:
\begin{lemma}\cite{BiharmonicWGReOrder}
  There are positive constants $C$ and $\lambda$ such that
  \begin{align}
    \sumT \normT{\G v_0}^2\leq C\norm{v}_{2,h}^2,\quad \forall v\in V_h^0,
  \end{align}
  and
  \begin{align}
    \sumT \normT{\G v_0}^2\leq \lambda h^{-2}\norm{v}^2 + \frac{C}{\lambda}h^2\norm{v}_{2,h}^2,\quad\forall v\in V_h^0.
  \end{align}
\end{lemma}
Further, using the projection inequality and the Lemma \ref{lemma-norm-equ}, we have
\begin{align*}
  \Big(\sumT \normPT{v_0-Q_bv_0}^2\Big)^{\frac{1}{2}}\leq Ch^{\frac{1}{2}}\Big(\sumT \normT{\G v_0}^2\Big)^{\frac{1}{2}},
\end{align*}
which implies
\begin{align}
  \label{v0-Qbv01}\Big(\sumT \normPT{v_0-Q_bv_0}^2\Big)^{\frac{1}{2}}\leq &Ch^{\frac{1}{2}}\trb{v},\quad\forall v\in V_h^0,\\
  \label{v0-Qbv02}\Big(\sumT \normPT{v_0-Q_bv_0}^2\Big)^{\frac{1}{2}}\leq &\lambda ^{\frac{1}{2}}Ch^{-\frac{1}{2}}\norm{v}+\lambda ^{-\frac{1}{2}}Ch^{\frac{3}{2}}\trb{v},\quad\forall v\in V_h^0.
\end{align}

\begin{lemma}
  For any $v\in V_h^0$, if $w\in H^{k+2}(\Omega)$, we have
  \begin{align}
    \label{l1-est1}|l_1(w,v)|\leq &Ch^k\norm{w}_{k+2}\trb{v},\\
    \label{l2-est1}|l_2(w,v)|\leq &Ch^k\norm{w}_{k+2}\trb{v},\\
    \label{l3-est1}|l_3(w,v)|\leq &Ch^{k-1}\norm{w}_{k+2}\trb{v}\\
    |l_3(w,v)|\leq &\lambda ^{\frac{1}{2}}Ch^{k-2}\norm{w}_{k+2}\norm{v}+\lambda ^{-\frac{1}{2}}Ch^{k}\norm{w}_{k+2}\trb{v}.
  \end{align}
  Specially, when $w\in H^4(\Omega)$, we get
  \begin{align}
    \label{l1-est2}|l_1(w,v)|\leq &Ch^2\norm{w}_{4}\trb{v},\\
    \label{l2-est2}|l_2(w,v)|\leq &Ch^2\norm{w}_{4}\trb{v},\\
    \label{l3-est2}|l_3(w,v)|\leq &Ch\norm{w}_{4}\trb{v},\\
    \label{l3-est3}|l_3(w,v)|\leq &\lambda ^{\frac{1}{2}}C\norm{w}_4\norm{v}+\lambda ^{-\frac{1}{2}}Ch^2\norm{w}_{4}\trb{v}
  \end{align}
\end{lemma}
\begin{proof}
  From the Cauchy-Schwarz inequality and (\ref{Qh-est-h1})-(\ref{Qh-est-h3}), we can derive
  \begin{align*}
    |l_1(w,v)|=&|\sumT \langle \G (\D w-\dQ_h\D w)\cdot\bn,Q_b(v_0-v_b)\rangle _{\partial T}|\\
    \leq &\Big(\sumT h_T^3\normPT{\G (\D w-\dQ_h\D w)}^2\Big)^{\frac{1}{2}}\Big(\sumT h_T^{-3}\normPT{Q_bv_0-v_b}^2\Big)^{\frac{1}{2}}\\
    \leq &Ch^k\norm{w}_{k+2}\norm{v}_{2,h}\\
    \leq &Ch^k\norm{w}_{k+2}\trb{v},
  \end{align*}
  \begin{align*}
    |l_2(w,v)|=&|\sumT \langle \D w-\dQ_h\D w,(\G v_0-v_n\bn_e)\cdot\bn\rangle _{\partial T}|\\
    \leq &\Big(\sumT h_T\normPT{\D w-\dQ_h\D w}^2\Big)^{\frac{1}{2}}\Big(\sumT h_T^{-1}\normPT{(\G v_0-v_n\bn_e)\cdot\bn}\Big)^{\frac{1}{2}}\\
    \leq &Ch^k\norm{w}_{k+2}\norm{v}_{2,h}\\
    \leq &Ch^k\norm{w}_{k+2}\trb{v}.
  \end{align*}

  Since $\langle \G (Q_0\D w)\cdot\bn,v_0-Q_bv_0\rangle _{\partial \T _h}=0$, the Cauchy-Schwarz inequality, the trace inequality, the projection inequality and (\ref{v0-Qbv01}), we obtain
  \begin{align*}
    |l_3(w,v)|=&|\langle \G (\D w)\cdot\bn,v_0-Q_bv_0\rangle _{\p\T_h}|\\
    = &|\langle \G (\D w-Q_0\D w)\cdot\bn,v_0-Q_bv_0\rangle _{\p\T_h}|\\
    \leq &Ch^{-\frac{3}{2}}\Big(\sumT h_T^3\normPT{\G (\D w-Q_0\D w)}^2\Big)^{\frac{1}{2}}\Big(\sumT \normPT{v_0-Q_bv_0}^2\Big)^{\frac{1}{2}}\\
    \leq &Ch^{-\frac{3}{2}}\Big(\sumT (h_T^2\normT{\G (\D w-Q_0\D w)}^2+h_T^4\normT{\G ^2(\D w-Q_0\D w)}^2)\Big)^{\frac{1}{2}} Ch^{\frac{1}{2}}\trb{v}\\
    \leq &Ch^{-\frac{3}{2}}Ch^{k}\norm{w}_{k+2}Ch^{\frac{1}{2}}\trb{v}\\
    \leq &Ch^{k-1}\norm{w}_{k+2}\trb{v}.
  \end{align*}

  If we use (\ref{v0-Qbv02}) instead of (\ref{v0-Qbv01}), we have
  \begin{align*}
    |l_3(w,v)|\leq &Ch^{-\frac{3}{2}}Ch^{k}\norm{w}_{k+2}(\lambda ^{\frac{1}{2}}Ch^{-\frac{1}{2}}\norm{v}+\lambda ^{-\frac{1}{2}}Ch^{\frac{3}{2}}\trb{v})\\
    \leq &\lambda ^{\frac{1}{2}}Ch^{k-2}\norm{w}_{k+2}\norm{v}+\lambda ^{-\frac{1}{2}}Ch^{k}\norm{w}_{k+2}\trb{v}.
  \end{align*}

  Let $k=2$, and we get the (\ref{l1-est2})-(\ref{l3-est3}).
\end{proof}

\subsection{Error estimates}
\begin{lemma}
  If $w\in H^{k+2}(\Omega)$, we have
  \begin{align}
    \label{trb-Qhw-w}\trb{Q_hw-w}\leq Ch^{k-1}\norm{w}_{k+2}.
  \end{align}
\end{lemma}
\begin{proof}
  Using the definition of the $\DW$, the integration by parts, 
  \begin{align*}
    (\DW (Q_hw-w),q)_T=&(\D (Q_0w-w),q)_T+\langle Q_b(Q_0w-w-(Q_bw-w)),\G q\cdot\bn\rangle _{\p T}\\
    &-\langle \G (Q_0w-w)\cdot\bn-(Q_n(\G w\cdot\bn_e)\bn_e\cdot\bn-\G w\cdot\bn),q\rangle _{\p T}\\
    =&(Q_0w-w,\D q)_T-\langle Q_0w-w,\G q\cdot\bn\rangle _{\p T}+\langle \G (Q_0w-w)\cdot\bn,q\rangle _{\p T}\\
    &+\langle Q_b(Q_0w-Q_bw),\G q\cdot\bn\rangle _{\p T}-\langle \G (Q_0w-w)\cdot\bn-(Q_n(\G w\cdot\bn)-\G w\cdot\bn),q\rangle _{\p T}\\
    =&(Q_0w-w,\D q)_T+\langle Q_b(Q_0w-Q_bw)-(Q_0w-w),\G q\cdot\bn\rangle _{\p T}+\langle Q_n(\G w\cdot\bn)-\G w\cdot\bn,q\rangle _{\p T}\\
    \leq &\normT{Q_0w-w}\normT{\D q}+(\normPT{Q_b(Q_0w-w)}+\normPT{Q_0w-w})\normPT{\G q}\\
    &+\normPT{Q_n (\G w\cdot\bn)-\G w\cdot\bn}\normPT{q}\\
    \leq &Ch^{k+1}|w|_{k+1,T}h^{-2}\normT{q}+2\normPT{Q_0w-w}Ch^{-\frac{3}{2}}\normT{q}+Ch^{-\frac{1}{2}}h^k|\G w\cdot\bn|_{k,T}h^{-\frac{1}{2}}\normT{q}\\
    \leq &Ch^{k-1}|w|_{k+1,T}\normT{q}+Ch^{-\frac{1}{2}}h^{k+1}|w|_{k+1,T}h^{-\frac{3}{2}}\normT{q}+Ch^{k-1}|w|_{k+1,T}\normT{q}\\
    \leq &Ch^{k-1}|w|_{k+1,T}\normT{q}
  \end{align*}
  hold true for any $q\in P_j[T]$. Then we choose $q=\DW (Q_hw-w)$, (\ref{trb-Qhw-w}) is proven.
\end{proof}

\begin{theorem}
  Suppose $u\in H^{k+2}(\Omega)$, and we have
  \begin{align}
    \label{trb-Qhu-uh}\trb{Q_hu-u_h}\leq &Ch^{k-1}\norm{u}_{k+2},\\
    \label{trb-u-uh}\trb{u-u_h}\leq &Ch^{k-1}\norm{u}_{k+2}.
  \end{align}
\end{theorem}
\begin{proof}
  Setting $v=Q_hu-u_h$ in (\ref{err-equ2}), and using (\ref{l1-est1})-(\ref{l3-est1}) and (\ref{trb-Qhw-w}), we find
  \begin{align*}
    \trb{\varepsilon _h}^2\leq &|l_1(u,\varepsilon _h)|+|l_2(u,\varepsilon _h)|+|l_3(u,\varepsilon _h)|+\trb{Q_hu-u}\trb{\varepsilon _h}\\
    \leq & Ch^k\norm{u}_{k+2}\trb{\varepsilon _h}+Ch^{k-1}\norm{u}_{k+2}\trb{\varepsilon _h}\\
    \leq &Ch^{k-1}\norm{u}_{k+2}\trb{\varepsilon _h},
  \end{align*}
  which implies (\ref{trb-Qhu-uh}). This yields
  \begin{align*}
    \trb{u-u_h} \leq &\trb{u-Q_hu}+\trb{Q_hu-u_h}\\
    \leq &Ch^{k-1}\norm{u}_{k+2}.
  \end{align*}
\end{proof}

Consider the following dual problem:
  \begin{align}
    \label{dual-equ1}\Delta^2\phi&=\varepsilon _0, \qquad in~\Omega,\\
    \label{dual-equ2}\phi&=0,\qquad on ~\partial \Omega,\\
    \label{dual-equ3}\pa{\phi}{\bn}&=0,\qquad on ~\partial \Omega,
    \end{align}
And assume $\norm{\phi}_4\leq C\norm{\varepsilon _0}$, we get the $L^2$ norm of the error as follows.
\begin{theorem}
  Suppose $u\in H^{k+2}(\Omega)$, and we obtained
  \begin{align}
    \label{L2-Q0u-u0}\norm{Q_0u-u_0}\leq &~Ch^{k+k_0-2}\norm{u}_{k+2},\\
    \label{L2-u-u0}\norm{u-u_0}\leq &~Ch^{k+k_0-2}\norm{u}_{k+2},
  \end{align}
  where $k_0=\min{\{k,3\}}$.
\end{theorem}
\begin{proof}
  \begin{align*}
    \norm{\varepsilon _0}^2=&(\D ^2\phi,\varepsilon _0)_{\T _h}\\
    =&(\DW \phi,\DW \varepsilon _h)_{\T _h}+l_1(\phi,\varepsilon _h)-l_2(\phi ,\varepsilon _h)+l_3(\phi ,\varepsilon _h)\\
    =&(\DW Q_h\phi,\DW\varepsilon _h)_{\T _h}+(\DW (\phi-Q_h\phi),\DW \varepsilon _h)_{\T _h}+l_1(\phi ,\varepsilon _h)-l_2(\phi ,\varepsilon _h)+l_3(\phi ,\varepsilon _h)\\
    =&-l_1(u,Q_h\phi)+l_2(u,Q_h\phi)-l_3(u,Q_h\phi)+(\DW (Q_hu-u),\DW Q_h\phi)_{\T _h}\\
    &+(\DW (\phi-Q_h\phi),\DW \varepsilon _h)_{\T _h}+l_1(\phi ,\varepsilon _h)-l_2(\phi ,\varepsilon _h)+l_3(\phi ,\varepsilon _h)\\
    =&I_1+I_2+I_3+I_4+I_5+I_6+I_7+I_8.
  \end{align*}
  For $I_1$, $I_2$ and $I_3$, using the Cauchy-Schwarz inequality, (\ref{Qh-est-h1})-(\ref{Qh-est-h3}), the definition of the projection operator $Q_b$ and $Q_n$, the trace inequality and the projection inequality, we can get
  \begin{align*}
    |I_1|=&|\langle \G (\D u-\dQ_h\D u)\cdot\bn,Q_b(Q_0\phi-Q_b\phi)\rangle _{\p\T_h}|\\
    \leq &\Big(\sumT h_T^3\normPT{\G (\D u-\dQ_h\D u)}^2\Big)^{\frac{1}{2}}\Big(\sumT h_T ^{-3}\normPT{Q_b(Q_0\phi-Q_b\phi)}^2\Big)^{\frac{1}{2}}\\
    \leq &Ch^k\norm{u}_{k+2}\Big(\sumT h_T ^{-3}\normPT{Q_0\phi-\phi}^2\Big)^{\frac{1}{2}}\\
    \leq &Ch^k\norm{u}_{k+2}\Big(\sumT (h_T ^{-4}\normT{Q_0\phi-\phi}^2+h_T^{-2}\normT{\G(Q_0\phi-\phi)}^2)\Big)^{\frac{1}{2}}\\
    \leq &Ch^k\norm{u}_{k+2}Ch^{k_0-1}\norm{\phi}_4\\
    \leq &Ch^{k+k_0-1}\norm{u}_{k+2}\norm{\phi}_4,
  \end{align*}
  \begin{align*}
    |I_2|=&|\langle \D u-\dQ_h\D u,(\G Q_0\phi-Q_n(\G\phi\cdot\bn_e)\bn_e)\cdot\bn\rangle _{\p\T_h}|\\
    \leq &\Big(\sumT h_T\normPT{\D u-\dQ_h\D u}^2\Big)^{\frac{1}{2}}\Big(\sumT h_T^{-1}\normPT{\G (Q_0\phi)\cdot\bn-Q_n(\G\phi\cdot\bn)}^2\Big)^{\frac{1}{2}}\\
    \leq &Ch^k\norm{u}_{k+2}\Big(\sumT h_T^{-1}\normPT{\G (Q_0\phi-\phi)\cdot\bn}\Big)^{\frac{1}{2}}\\
    \leq &Ch^k\norm{u}_{k+2}\Big(\sumT h_T^{-1}\normPT{\G (Q_0\phi-\phi)}\Big)^{\frac{1}{2}}\\
    \leq &Ch^k\norm{u}_{k+2}h^{-2}Ch^{k_0+1}\norm{\phi}_4\\
    \leq &Ch^{k+k_0-1}\norm{u}_{k+2}\norm{\phi}_4,
  \end{align*}
  \begin{align*}
    |I_3|=&|\langle \G (\D u)\cdot\bn,Q_0\phi-Q_bQ_0\phi\rangle _{\p\T_h}|\\
    =&|\langle \G (\D u-Q_0\D u)\cdot\bn,Q_0\phi-Q_bQ_0\phi\rangle _{\p\T_h}|\\
    =&|\langle \G (\D u-Q_0\D u)\cdot\bn,Q_0\phi-\phi +\phi -Q_b\phi +Q_b\phi -Q_bQ_0\phi\rangle _{\p\T_h}|\\
    =&|\langle \G (\D u-Q_0\D u)\cdot\bn,Q_0\phi-\phi\rangle _{\p\T_h}+\langle \G (\D u-Q_0\D u)\cdot\bn,Q_b\phi -Q_bQ_0\phi\rangle _{\p\T_h}|\\
    \leq &Ch^{-\frac{3}{2}}\Big(\sumT h_T^3\normPT{\G (\D u-Q_0\D u)}^2\Big)^{\frac{1}{2}}\Big(\sumT \normPT{Q_0\phi-\phi}^2\Big)^{\frac{1}{2}}\\
    \leq &Ch^{-\frac{3}{2}}Ch^k\norm{\D u}_{k}Ch^{-\frac{1}{2}}Ch^{k_0+1}\norm{\phi}_4\\
    \leq &Ch^{k+k_0-1}\norm{u}_{k+2}\norm{\phi}_4,
  \end{align*}
  where we utilize
  \begin{align*}
    \langle \G (\D u-Q_0\D u)\cdot\bn,\phi -Q_b\phi\rangle _{\p\T_h}=\langle \G (\D u)\cdot\bn,\phi -Q_b\phi\rangle _{\p\T_h}-\langle \G (Q_0\D u)\cdot\bn,\phi -Q_b\phi\rangle _{\p\T_h}=0.
  \end{align*}
  
  If we assume $P_1$ is the $L^2$ projection to the polynomials, there exists the following result
  \begin{align*}
    &(\DW (Q_hu-u),P_1\D\phi)_{\T _h}\\
    =&(\D (Q_0u-u),P_1\D\phi)_{\T _h}+\langle Q_b[Q_0u-u-(Q_bu-u)],\G(P_1\D\phi)\cdot\bn\rangle _{\p\T_h}\\
    &-\langle \G (Q_0u-u)\cdot\bn-Q_n (\G u\cdot\bn_e)\bn_e\cdot\bn+\G u\cdot\bn,P_1\D\phi\rangle _{\p\T_h}\\
    =&-(\G(Q_0u-u),\G P_1\D \phi)_{\T_h}+\langle \G (Q_0u-u)\cdot\bn,P_1\D\phi\rangle _{\p\T_h}+\langle Q_b(Q_0u-Q_bu),\G (P_1\D \phi)\cdot\bn\rangle _{\p\T _h}\\
    &-\langle \G (Q_0u-u)\cdot\bn-Q_n (\G u\cdot\bn)+\G u\cdot\bn,P_1\D\phi\rangle _{\p\T_h}\\
    =&(Q_0u-u,\D P_1\D \phi)_{\T_h}-\langle Q_0u-u,\G P_1\D\phi\cdot\bn\rangle _{\p\T_h}+\langle Q_b(Q_0u-Q_bu),\G (P_1\D \phi)\cdot\bn\rangle _{\p\T _h}\\
    &+\langle Q_n(\G u\cdot\bn)-\G u\cdot\bn,P_1\D\phi\rangle _{\p\T_h}\\
    =&0-\langle Q_0u-u,\G P_1\D\phi\cdot\bn\rangle _{\p\T_h}+\langle Q_0u-u,\G (P_1\D \phi)\cdot\bn\rangle _{\p\T _h}+0\\
    =&0.
  \end{align*}
  By the above equation, the (\ref{operator-ex}), the Cauchy-Schwarz inequality, the definition of $\dQ_h$, (\ref{trb-Qhw-w}) and the projection inequality, we know
  \begin{align*}
    |I_4|=&|(\DW (Q_hu-u),\DW Q_h\phi)_{\T_h}|\\
    \leq &|(\DW (Q_hu-u),\DW (Q_h\phi-\phi))_{\T_h}|+|(\DW (Q_hu-u),\DW \phi)_{\T_h}|\\
    \leq &\trb{Q_hu-u}\trb{Q_h\phi -\phi}+|(\DW (Q_hu-u),\dQ_h\D \phi)_{\T_h}|\\
    \leq &Ch^{k-1}\norm{u}_{k+2}Ch^{k_0-1}\norm{\phi}_4+|(\DW (Q_hu-u),\D \phi)_{\T_h}|\\
    \leq &Ch^{k+k_0-2}\norm{u}_{k+2}\norm{\phi}_4+|(\DW (Q_hu-u),\D \phi-P_1\D \phi)_{\T_h}|\\
    \leq &Ch^{k+k_0-2}\norm{u}_{k+2}\norm{\phi}_4+\trb{Q_hu-u}\norm{\D\phi -P_1\D \phi}\\
    \leq &Ch^{k+k_0-2}\norm{u}_{k+2}\norm{\phi}_4+Ch^{k-1}\norm{u}_{k+2}Ch^2\norm{\phi}_4\\
    \leq &Ch^{k+k_0-2}\norm{u}_{k+2}\norm{\phi}_4.
  \end{align*}

  From (\ref{trb-Qhw-w}), (\ref{trb-Qhu-uh}) and (\ref{l1-est2})-(\ref{l3-est3}), we find
  \begin{align*}
    |I_5|=&|(\DW (\phi-Q_h\phi),\DW \varepsilon _h)_{\T_h}|\\
    \leq &\trb{\phi -Q_h\phi}\trb{\varepsilon _h}\\
    \leq &Ch^{k_0-1}\norm{\phi}_4Ch^{k-1}\norm{u}_{k+2}\\
    \leq &Ch^{k+k_0-2}\norm{u}_{k+2}\norm{\phi}_4,
  \end{align*}
  \begin{align*}
    |I_6|=&|l_1(\phi ,\varepsilon _h)|\\
    \leq &Ch^2\norm{\phi}_4\trb{\varepsilon _h}\\
    \leq &Ch^2\norm{\phi}_4Ch^{k-1}\norm{u}_{k+2}\\
    \leq &Ch^{k+1}\norm{u}_{k+2}\norm{\phi}_4,
  \end{align*}
  \begin{align*}
    |I_7|=&|l_2(\phi,\varepsilon _h)|\\
    \leq &Ch^2\norm{\phi}_4\trb{\varepsilon _h}\\
    \leq &Ch^2\norm{\phi}_4Ch^{k-1}\norm{u}_{k+2}\\
    \leq &Ch^{k+1}\norm{u}_{k+2}\norm{\phi}_4,
  \end{align*}
  \begin{align*}
    |I_8|=&|l_3(\phi,\varepsilon _h)|\\
    \leq &\lambda ^{\frac{1}{2}}C\norm{\phi}_4\norm{\varepsilon _h}+\lambda ^{-\frac{1}{2}}Ch^2\norm{\phi}_4\trb{\varepsilon _h}\\
    \leq &\lambda ^{\frac{1}{2}}C\norm{\phi}_4\norm{\varepsilon _0}+\lambda ^{-\frac{1}{2}}Ch^{k+1}\norm{\phi}_4\norm{u}_{k+2}.
  \end{align*}

  Therefore, we have
  \begin{align*}
    \norm{\varepsilon _0}^2\leq &\lambda ^{\frac{1}{2}}C\norm{\phi}_4\norm{\varepsilon _0}+\lambda ^{-\frac{1}{2}}Ch^{k+1}\norm{\phi}_4\norm{u}_{k+2}+Ch^{k+1}\norm{u}_{k+2}\norm{\phi}_4+Ch^{k+k_0-2}\norm{u}_{k+2}\norm{\phi}_4\\
    &+Ch^{k+k_0-1}\norm{u}_{k+2}\norm{\phi}_4,
  \end{align*}
  Let $\lambda$ such that $\lambda ^{\frac{1}{2}}C\norm{\phi}_4\leq \frac{1}{2}\norm{\varepsilon _0}$, then we get
  \begin{align*}
    \norm{\varepsilon _0}^2\leq &\frac{1}{2}\norm{\varepsilon _0}^2+Ch^{k+1}\norm{\phi}_4\norm{u}_{k+2}+Ch^{k+1}\norm{u}_{k+2}\norm{\phi}_4+Ch^{k+k_0-2}\norm{u}_{k+2}\norm{\phi}_4\\
    &+Ch^{k+k_0-1}\norm{u}_{k+2}\norm{\phi}_4,\\
    \leq &\frac{1}{2}\norm{\varepsilon _0}^2+Ch^{k+1}\norm{\varepsilon _0}\norm{u}_{k+2}+Ch^{k+1}\norm{u}_{k+2}\norm{\varepsilon _0}+Ch^{k+k_0-2}\norm{u}_{k+2}\norm{\varepsilon _0}\\
    &+Ch^{k+k_0-1}\norm{u}_{k+2}\norm{\varepsilon _0},
  \end{align*}
  which implies
  \begin{align*}
    \norm{\varepsilon _0}\leq Ch^{k+k_0-2}\norm{u}_{k+2}.
  \end{align*}
\end{proof}

  %======================================================================================
\section{Numerical Experiments}

In this section, we shall utilize examples to verify the rationality of the theoretical results.
\subsection{Example 1}
We consider the square domain $\Omega = (0,1)^2$ and the exact solution as follows 
\begin{align}\label{ex1}
  u = x^2(1-x)^2y^2(1-y)^2.
\end{align}

On triangular meshes, we set $j=k+2$. When $k=2$, the convergence orders appear as $O(h^1)$ under the $H^2$ norm and $O(h^2)$ under the $L^2$ norm in Table \ref{ex1_tri_P2}. The convergence orders for $k=3$ in the $H^2$ and $L^2$ norms are of orders $O(h^2)$ and $O(h^4)$ in Table \ref{ex1_tri_P3}. And the results shown in Tables \ref{ex1_tri_P2}-\ref{ex1_tri_P3} coincide with the theorems in the previous section.

\begin{table}[htbp]
  \centering
  \caption{Error values and convergence rates for (\ref{ex1}) on triangular meshes with $k=2$}\label{ex1_tri_P2}
    \begin{tabular}{c|c|c|c|c|c|c}
    \hline
    $n$     & $\trb{u-u_h}$ & Rate  & $\norm{u-u_h}_{2,h}$ & Rate  & $\norm{u-u_0}$ & Rate \\
    \hline
    8     & 9.5700E-03 & ---   & 1.5165E-02 & ---   & 4.6178E-05 & --- \\
    16    & 4.7707E-03 & 1.00  & 7.5996E-03 & 1.00  & 1.1483E-05 & 2.01  \\
    32    & 2.3665E-03 & 1.01  & 3.7959E-03 & 1.00  & 2.8221E-06 & 2.02  \\
    64    & 1.1760E-03 & 1.01  & 1.8959E-03 & 1.00  & 6.9606E-07 & 2.02  \\
    128   & 5.8582E-04 & 1.01  & 9.4727E-04 & 1.00  & 1.7229E-07 & 2.01  \\
    \hline
    \end{tabular}%
\end{table}%

\begin{table}[htbp]
  \centering
  \caption{Error values and convergence rates for (\ref{ex1}) on triangular meshes with $k=3$}\label{ex1_tri_P3}
    \begin{tabular}{c|c|c|c|c|c|c}
    \hline
    $n$     & $\trb{u-u_h}$ & Rate  & $\norm{u-u_h}_{2,h}$ & Rate  & $\norm{u-u_0}$ & Rate \\
    \hline
    2     & 1.6283E-02 & ---   & 2.5831E-02 & ---   & 1.4332E-04 & ---\\
    4     & 4.0486E-03 & 2.01  & 7.0678E-03 & 1.87  & 9.8818E-06 & 3.86  \\
    8     & 1.0417E-03 & 1.96  & 1.8704E-03 & 1.92  & 6.7580E-07 & 3.87  \\
    16    & 2.6142E-04 & 1.99  & 4.7555E-04 & 1.98  & 4.2635E-08 & 3.99  \\
    32    & 6.5097E-05 & 2.01  & 1.1948E-04 & 1.99  & 2.6431E-09 & 4.01  \\
    \hline
    \end{tabular}%
\end{table}%

On polygonal meshes, let $j=k+4$ and Tables \ref{ex1_poly_P2}-\ref{ex1_poly_P3} show the the convergence rates coincident with our theoretical analysis.

\begin{table}[htbp]
  \centering
  \caption{Error values and convergence rates for (\ref{ex1}) on polygonal meshes with $k=2$}\label{ex1_poly_P2}
    \begin{tabular}{c|c|c|c|c|c|c}
    \hline
    $n$     & $\trb{u-u_h}$ & Rate  & $\norm{u-u_h}_{2,h}$ & Rate  &$\norm{u-u_0}$ & Rate \\
    \hline
    16    & 1.2242E-02 & ---   & 1.0176E-02 & ---   & 7.4200E-05 & --- \\
    32    & 6.5755E-03 & 0.90  & 5.2017E-03 & 0.97  & 2.1480E-05 & 1.79  \\
    64    & 3.3740E-03 & 0.96  & 2.6315E-03 & 0.98  & 5.6675E-06 & 1.92  \\
    128   & 1.7046E-03 & 0.99  & 1.3245E-03 & 0.99  & 1.4475E-06 & 1.97  \\
    256   & 8.5618E-04 & 0.99  & 6.6462E-04 & 0.99  & 3.6900E-07 & 1.97  \\
    \hline
    \end{tabular}%
\end{table}%

\begin{table}[htbp]
  \centering
  \caption{Error values and convergence rates for (\ref{ex1}) on polygonal meshes with $k=3$}\label{ex1_poly_P3}
    \begin{tabular}{c|c|c|c|c|c|c}
    \hline
    $n$     & $\trb{u-u_h}$ & Rate  & $\norm{u-u_h}_{2,h}$ & Rate  & $\norm{u-u_0}$ & Rate \\
    \hline
    4     & 6.9507E-03 & ---   & 8.4116E-03 & ---   & 2.5430E-05 & --- \\
    8     & 2.2107E-03 & 1.65  & 2.5972E-03 & 1.70  & 2.5638E-06 & 3.31  \\
    16    & 6.1553E-04 & 1.84  & 7.1587E-04 & 1.86  & 2.0686E-07 & 3.63  \\
    32    & 1.5974E-04 & 1.95  & 1.8464E-04 & 1.96  & 1.4303E-08 & 3.85  \\
    64    & 4.0278E-05 & 1.99  & 4.6334E-05 & 1.99  & 8.2573E-10 & 4.11  \\
    \hline
    \end{tabular}%
\end{table}%

%%%%%%%%%%%%%%%%%%%%%%%%%%%%%%%%%%%%%%%%%%%%%%%%%%%%%%%%%%

%%%%%%%%%%%%%%%%%%%%%%%%%%%%%%%%%%%%%%%%%%%%%%%%%%%%%%%%%%
\subsection{Example 2}
We choose the same solution area as in the above example and the exact solution is
\begin{align}\label{ex2}
  u = \sin {(\pi x)}\sin{(\pi y)}.
\end{align}  
Set $j=k+2$ on triangular meshes and $j=k+4$ on polygonal meshes, then the related results are shown in Tables \ref{ex2_tri_P2}-\ref{ex2_tri_P3} and \ref{ex2_poly_P2}-\ref{ex2_poly_P3}, respectively. These convergence orders agree with the theoretical results.

\begin{table}[htbp]
  \centering
  \caption{Error values and convergence rates for (\ref{ex2}) on triangular meshes with $k=2$}\label{ex2_tri_P2}
    \begin{tabular}{c|c|c|c|c|c|c}
    \hline
    $n$     & $\trb{u-u_h}$ & Rate  & $\norm{u-u_h}_{2,h}$ & Rate  & $\norm{u-u_0}$ & Rate \\
    \hline
    8     & 9.4121E-01 & ---   & 1.3644E+00 & ---   & 3.4904E-03 & --- \\
    16    & 4.7017E-01 & 1.00  & 6.8351E-01 & 1.00  & 8.9894E-04 & 1.96  \\
    32    & 2.3481E-01 & 1.00  & 3.4193E-01 & 1.00  & 2.2798E-04 & 1.98  \\
    64    & 1.1732E-01 & 1.00  & 1.7099E-01 & 1.00  & 5.7391E-05 & 1.99  \\
    128   & 5.8633E-02 & 1.00  & 8.5497E-02 & 1.00  & 1.4433E-05 & 1.99  \\
    \hline
    \end{tabular}%
\end{table}%

\begin{table}[htbp]
  \centering
  \caption{Error values and convergence rates for (\ref{ex2}) on triangular meshes with $k=3$}\label{ex2_tri_P3}
    \begin{tabular}{c|c|c|c|c|c|c}
    \hline
    $n$     & $\trb{u-u_h}$ & Rate  & $\norm{u-u_h}_{2,h}$ & Rate  & $\norm{u-u_0}$ & Rate \\
    \hline
    2     & 1.0474E+00 & ---   & 1.6703E+00 & ---   & 8.5304E-03 & --- \\
    4     & 2.7187E-01 & 1.95  & 4.5396E-01 & 1.88  & 5.6923E-04 & 3.91  \\
    8     & 6.8241E-02 & 1.99  & 1.1647E-01 & 1.96  & 3.6278E-05 & 3.97  \\
    16    & 1.7027E-02 & 2.00  & 2.9360E-02 & 1.99  & 2.2723E-06 & 4.00  \\
    32    & 4.2486E-03 & 2.00  & 7.3603E-03 & 2.00  & 1.4099E-07 & 4.01  \\
    \hline
    \end{tabular}%
\end{table}%

\begin{table}[htbp]
  \centering
  \caption{Error values and convergence rates for (\ref{ex2}) on polygonal meshes with $k=2$}\label{ex2_poly_P2}
    \begin{tabular}{c|c|c|c|c|c|c}
    \hline
    $n$     & $\trb{u-u_h}$ & Rate  & $\norm{u-u_h}_{2,h}$ & Rate  & $\norm{u-u_0}$ & Rate \\
    \hline
    8     & 2.9734E+00 & ---   & 1.8766E+00 & ---   & 1.2437E-02 & --- \\
    16    & 1.5486E+00 & 0.94  & 9.9298E-01 & 0.92  & 3.7323E-03 & 1.74  \\
    32    & 7.8404E-01 & 0.98  & 5.0677E-01 & 0.97  & 9.8116E-04 & 1.93  \\
    64    & 3.9390E-01 & 0.99  & 2.5551E-01 & 0.99  & 2.4850E-04 & 1.98  \\
    128   & 1.9736E-01 & 1.00  & 1.2822E-01 & 0.99  & 6.2082E-05 & 2.00  \\
    \hline
    \end{tabular}%
\end{table}%

\begin{table}[htbp]
  \centering
  \caption{Error values and convergence rates for (\ref{ex2}) on polygonal meshes with $k=3$}\label{ex2_poly_P3}
    \begin{tabular}{c|c|c|c|c|c|c}
    \hline
    $n$     & $\trb{u-u_h}$ & Rate  & $\norm{u-u_h}_{2,h}$ & Rate  & $\norm{u-u_0}$ & Rate \\
    \hline
    4     & 5.2002E-01 & ---   & 5.9866E-01 & ---   & 1.1222E-03 & --- \\
    8     & 1.3220E-01 & 1.98  & 1.6706E-01 & 1.84  & 8.8280E-05 & 3.67  \\
    16    & 3.3449E-02 & 1.98  & 4.3703E-02 & 1.93  & 6.1111E-06 & 3.85  \\
    32    & 8.4305E-03 & 1.99  & 1.1141E-02 & 1.97  & 3.9362E-07 & 3.96  \\
    64    & 2.1173E-03 & 1.99  & 2.8113E-03 & 1.99  & 1.7194E-08 & 4.52  \\
    \hline
    \end{tabular}%
\end{table}%

%%%%%%%%%%%%%%%%%%%%%%%%%%%%%%%%%%%%%%%%%%%%%%%%%%%%%%%%%%
% \subsection{Example 3}
% We choose the same solution area and final time as in the above examples. The exact solution is
% \begin{align}\label{ex3}
%   u = e^{x+y}.
% \end{align}
% The relevant outcomes are displayed in Tables \ref{ex3_tri} and \ref{ex3_poly}.
%%%%%%%%%%%%%%%%%%%%%%%%%%%%%%%%%%%%%%%%%%%%%%%%%%%%%%%%%%

%\clearpage
\newpage

%-----------------------------------------------------------------------------------------------
\section{Concluding remarks and ongoing work}

In the paper, we use weak functions $(P_k(T),P_{k-1}(e),P_{k-1}(e))$ to build the SFWG numerical scheme for the biharmonic equation. To achieve the optimal convergence orders of the errors, we modify the definition of the weak laplacian in \cite{BiharmonicSFWG}. Finally, the convergence rates in the $H^2$ and $L^2$ norms are of order $O(h^{k-1})$ and $O(h^{k+k_0-2})$, $(k_0=\min{\{k,3\}})$ which are verified by numerical examples.

In the future work, we will continue to study the WG related methods for fourth order equations.
%-----------------------------------------------------------------------------------------------

\bibliographystyle{siam}
\bibliography{library}

\begin{thebibliography}{10}

\bibitem{BiharmonicWG}
{\sc L.~{}~Mu, J.~Wang, and X.~Ye}, {\em Weak {Galerkin} finite element methods
  for the biharmonic equation on polytopal meshes}, Numerical Methods for
  Partial Differential Equations, 30 (2014), pp.~1003--1029.

\bibitem{CRMFEM}
{\sc P.~G. Ciarlet and P.-A. Raviart}, {\em A mixed finite element method for
  the biharmonic equation}, in Mathematical aspects of finite elements in
  partial differential equations ({P}roc. {S}ympos., {M}ath. {R}es. {C}enter,
  {U}niv. {W}isconsin, {M}adison, {W}is., 1974), Publication No. 33, Math. Res.
  Center, Univ. of Wisconsin-Madison, Academic Press, New York, 1974,
  pp.~125--145.

\bibitem{BiharmonicMWG}
{\sc M.~Cui, X.~Ye, and S.~Zhang}, {\em A modified weak galerkin finite element
  method for the biharmonic equation on polytopal meshes}, 3, pp.~91--105.

\bibitem{MR2817542}
{\sc J.~Hu, Y.~Huang, and S.~Zhang}, {\em The lowest order differentiable
  finite element on rectangular grids}, SIAM J. Numer. Anal., 49 (2011),
  pp.~1350--1368.

\bibitem{HJMFEM}
{\sc C.~Johnson}, {\em On the convergence of a mixed finite-element method for
  plate bending problems}, Numerische Mathematik,  (1973), pp.~43--62.

\bibitem{MR0386298}
{\sc T.~Miyoshi}, {\em A finite element method for the solutions of fourth
  order partial differential equations}, Kumamoto J. Sci. (Math.), 9 (1972/73),
  pp.~87--116.

\bibitem{BiharmonicMixFEM3}
{\sc P.~Monk}, {\em A mixed finite element method for the biharmonic equation},
  SIAM journal on numerical analysis, 24 (1987), pp.~737--749.

\bibitem{BiharmonicDGFEM}
{\sc I.~Mozolevski, E.~Süli, and P.~R. Bösing}, {\em hp-version a priori
  error analysis of interior penalty discontinuous galerkin finite element
  approximations to the biharmonic equation}, Journal of scientific computing,
  30 (2007), pp.~465--491.

\bibitem{BiharmonicWGMFEM}
{\sc L.~Mu, J.~Wang, Y.~Wang, and X.~Ye}, {\em A {Weak} {Galerkin} {Mixed}
  {Finite} {Element} {Method} for {Biharmonic} {Equations}}, Dec. 2012.
\newblock arXiv:1210.3818 [math].

\bibitem{BrinkmanWG}
{\sc L.~Mu, J.~Wang, and X.~Ye}, {\em A stable numerical algorithm for the
  brinkman equations by weak galerkin finite element methods}, 273,
  pp.~327--342.

\bibitem{mu_weak_2012}
\leavevmode\vrule height 2pt depth -1.6pt width 23pt, {\em Weak {Galerkin}
  {Finite} {Element} {Methods} on {Polytopal} {Meshes}}, Aug. 2012.
\newblock arXiv:1204.3655 [math].

\bibitem{PossionWGReOrder}
\leavevmode\vrule height 2pt depth -1.6pt width 23pt, {\em A weak {Galerkin}
  finite element method with polynomial reduction}, Journal of Computational
  and Applied Mathematics, 285 (2015), pp.~45--58.

\bibitem{MR1191139}
{\sc P.~Oswald}, {\em Hierarchical conforming finite element methods for the
  biharmonic equation}, SIAM J. Numer. Anal., 29 (1992), pp.~1610--1625.

\bibitem{StokesWG}
{\sc J.~Wang and X.~Ye}, {\em A weak galerkin finite element method for the
  stokes equations}, 42, pp.~155--174.
\newblock Publisher: Springer {US}.

\bibitem{PossionMixedWG}
{\sc J.~{}Wang and X.~Ye}, {\em A weak galerkin mixed finite element method for
  second order elliptic problems}, 83, pp.~2101--2126.

\bibitem{PossionWG}
{\sc J.~Wang and X.~Ye}, {\em A weak galerkin finite element method for
  second-order elliptic problems}, Journal of computational and applied
  mathematics, 241 (2013), pp.~103--115.

\bibitem{PossionMWG}
{\sc X.~Wang, N.~S. Malluwawadu, F.~Gao, and T.~C. {McMillan}}, {\em A modified
  weak galerkin finite element method}, 271, pp.~319--327.
\newblock Publisher: Elsevier B.V.

\bibitem{PossionSFWG}
{\sc X.~Ye and S.~Zhang}, {\em A stabilizer-free weak galerkin finite element
  method on polytopal meshes}, 371.
\newblock Publisher: Elsevier B.V.

\bibitem{BiharmonicSFWG}
{\sc X.~{}Ye and S.~Zhang}, {\em A stabilizer free weak galerkin method for the
  biharmonic equation on polytopal meshes}, 58, pp.~2572--2588.

\bibitem{PossionSFWGReOrder}
{\sc X.~Ye and S.~Zhang}, {\em A new weak gradient for the stabilizer free weak
  {Galerkin} method with polynomial reduction}, DCDS-B, 26 (2021), p.~4131.

\bibitem{BiharmonicWGReOrder}
{\sc R.~Zhang and Q.~Zhai}, {\em A weak galerkin finite element scheme for the
  biharmonic equations by using polynomials of reduced order}, 64,
  pp.~559--585.

\bibitem{BiharmonicCFEM2}
{\sc S.~Zhang}, {\em A c1-p2 finite element without nodal basis}, ESAIM -
  Mathematical Modelling and Numerical Analysis,  (2008), pp.~175--192.

\end{thebibliography}

\end{document}